\documentclass[11pt,leqno,twoside]{amsart}
\usepackage{amssymb,amsmath,amsthm,soul,color}
\usepackage{t1enc}
\usepackage[cp1250]{inputenc}
\usepackage{a4,indentfirst,latexsym}
\usepackage{graphics}
\usepackage{mathrsfs}
\usepackage{cite,enumitem,graphicx}
\usepackage[colorlinks=true,urlcolor=blue,
citecolor=red,linkcolor=blue,linktocpage,pdfpagelabels,
bookmarksnumbered,bookmarksopen]{hyperref}
\usepackage[english]{babel}
\usepackage[left=2.61cm,right=2.61cm,top=2.72cm,bottom=2.72cm]{geometry}
\usepackage[metapost]{mfpic}
\usepackage[hyperpageref]{backref}
\usepackage[colorinlistoftodos]{todonotes}
\usepackage[normalem]{ulem}
\allowdisplaybreaks

\makeatletter
\providecommand\@dotsep{5}
\def\listtodoname{List of Todos}
\def\listoftodos{\@starttoc{tdo}\listtodoname}
\makeatother

\numberwithin{equation}{section}
\newtheorem{Th}{Theorem}[section]
\newtheorem{Prop}[Th]{Proposition}
\newtheorem{Lem}[Th]{Lemma}
\newtheorem{Cor}[Th]{Corollary}

\newtheorem{Rem}[Th]{Remark}
\newtheorem{Ex}[Th]{Example}

\newenvironment{altproof}[1]
{\noindent
	{\em Proof of {#1}}.}
{\nopagebreak\mbox{}\hfill $\Box$\par\addvspace{0.5cm}}

\newcommand{\weakto}{\rightharpoonup}
\newcommand{\D}{\mathcal{D}^{1,2}(\RN)}
\newcommand{\Dr}{\mathcal{D}^{1,2}_{\cO(N)}(\RN)}
\newcommand{\Dnr}{\mathcal{D}^{1,2}_{\cO}(\RN)}

\newcommand{\DO}{\mathcal{D}^{1,2}_{\cO'}(\RN)}

   \newcommand{\vp}{\varphi}
   
   \newcommand{\eps}{\varepsilon}

   \def\supp{\mathrm{supp}}

   \def\N{\mathbb{N}}
   \def\R{\mathbb{R}}

   \def\K{\mathcal{K}}

   \def\RN{\mathbb{R}^N}
   
   \def\n{\nabla}
	\def\de{\partial}
\def\a{\alpha}
\def\b{\beta}

\def\s{\sigma}
\def\irn{\int_{\RN}}

\def\bl#1\el{\textcolor{blue}{#1}} 
\def\br#1\er{\textcolor{red}{#1}} 

   \newcommand{\cO}{{\mathcal O}}

\newcommand{\cC}{{\mathcal C}}
\newcommand{\cI}{{\mathcal I}}
\newcommand{\cJ}{{\mathcal J}}
\newcommand{\dis}{\displaystyle}

\title[Nonlinear scalar field equation]{Nonlinear scalar field equation\\with competing nonlocal terms}

\author[P. d'Avenia]{Pietro d'Avenia}
\author[J. Mederski]{Jaros\l aw Mederski}
\author[A. Pomponio]{Alessio Pomponio}

\address[P. d'Avenia and A. Pomponio]{\newline\indent
Dipartimento di Meccanica, Matematica e Management
\newline\indent 
Politecnico di Bari
\newline\indent
Via Orabona 4,  70125  Bari, Italy}
\email{\href{mailto:pietro.davenia@poliba.it}{pietro.davenia@poliba.it}}
\email{\href{mailto:alessio.pomponio@poliba.it}{alessio.pomponio@poliba.it}}

\address[J. Mederski]{ \newline\indent 
	Institute of Mathematics,\newline\indent 
	Polish Academy of Sciences \newline\indent 
	ul. \'Sniadeckich 8, 00-656 Warsaw, Poland \newline\indent 
	and\newline\indent  
	Departement of Mathematics, Institute for Analysis,\newline\indent 
	Karlsruhe Institute of Technology (KIT), \newline\indent 
	D-76128 Karlsruhe, Germany
}
\email{\href{mailto:jmederski@impan.pl}{jmederski@impan.pl}}

\thanks{P. d'Avenia and A. Pomponio are members of GNAMPA of INdAM and are partially supported by PRIN 2017JPCAPN {\em Qualitative and quantitative aspects of nonlinear PDEs}. P. d'Avenia is also supported by FRA2019 of Politecnico di Bari. J. Mederski was partially supported by the National Science Centre, Poland, Grant No. 2017/26/E/ST1/00817 and by the Deutsche Forschungsgemeinschaft (DFG, German Research Foundation) – Project-ID 258734477 – SFB 1173 during the stay at Karlsruhe Institute of Technology}

\date{\today}

\begin{document}
\begin{abstract} We find radial and nonradial solutions to the following nonlocal problem
	$$-\Delta u +\omega u= \big(I_\alpha\ast F(u)\big)f(u)-\big(I_\beta\ast G(u)\big)g(u) \text{ in } \mathbb{R}^N$$
	under general assumptions, in the spirit of Berestycki and Lions, imposed on $f$ and $g$, where $N\geq 3$, $0\leq \beta \leq \alpha<N$, $\omega\geq 0$,
	$f,g:\mathbb{R}\to \mathbb{R}$ are continuous functions with corresponding primitives $F,G$,
	and $I_\alpha,I_\beta$ are the Riesz potentials. If $\beta>0$, then we deal with two competing nonlocal terms modelling attractive and repulsive interaction potentials. 
\end{abstract}

\maketitle

\section{Introduction}

This paper mainly deals with the following problem
\begin{equation}\label{eq}
-\Delta u
= \big(I_\alpha\ast F(u)\big)f(u)-\big(I_\beta\ast G(u)\big)g(u)
\qquad \hbox{ in }\R^N,
\end{equation}
where $N\geq 3$, $0\le \b\le \a<N$, 
$f,g:\R\to \R$ are continuous functions with corresponding primitives
$$F(s)=\int_0^s f(t) dt,\quad
G(s)=\int_0^s g(t) dt.$$
Moreover $I_\gamma:\R^N\to\R$ is the Riesz potential 
$$I_\gamma(x):=\frac{\Gamma(\frac{N-\gamma}{2})}{\Gamma(\frac{\gamma}{2})\pi^{N/2}2^\gamma |x|^{N-\gamma}}\quad\hbox{for }x\in\R^N\setminus\{0\}\hbox{ and }\gamma\in (0,N),$$
while we set $I_0=\delta_0$, namely the identity for the convolution,where $\delta_0$ is the Dirac delta at zero.

If $N=3$, $\alpha=2$, $\beta=0$, $F(s)=\frac{1}{\sqrt{2}}|s|^2$ and $G(s)=s$, then \eqref{eq} is the well-known  Choquard, or Choquard-Pekar equation
\begin{equation*}
-\Delta u+u = \big(I_2\ast |u|^2\big)u\qquad \hbox{ in }\R^N.
\end{equation*} 
This equation comes, for instance, from an approximation to the Hartree-Fock theory of a
plasma \cite{Pekar,Lieb}   or it can be also derived in the framework of the Schr\"odinger-Newton system, (see \cite{MPT}). 
A variational approach for this case was presented by Lieb \cite{Lieb} and Lions \cite{Lions}. 

More generally, if $N\geq 3$, $F(s)=\frac{1}{\sqrt{p}}|s|^p$, for suitable $p$, $\alpha>0$ and $G(s)=s$, then weak solutions to \eqref{eq} can be obtained by means of critical points of the associated functional. If, for instance, $\frac{N+\alpha}{N}<p<\frac{N+\alpha}{N-2}$ and $\beta =0$,  then, according to the work of Moroz and Van Schaftingen \cite{MVJFA}, the Hardy-Littlewood-Sobolev inequality implies that $\big(I_\alpha\ast F(u)\big)F(u)\in L^1(\R^N)$ for $u\in H^1(\R^N)$. Moreover the associated functional is well-defined and of class $\cC^1$ on $H^1(\R^N)$,  and its  critical points correspond to solutions to
\begin{equation}\label{eqVM}
-\Delta u+u = \big(I_\alpha\ast F(u)\big)f(u)\qquad \hbox{ in }\R^N.
\end{equation} 
A ground state solution and its properties were obtained in \cite{MVJFA}. The same authors in \cite{MVTrans} also studied the existence of solutions with a general nonlinearity $F$ in the spirit of the classical result of Berestycki and Lions \cite{BL1}, namely 
\begin{equation}\label{eq:assumVM}
|sf(s)|\leq C\big(|s|^{\frac{N+\alpha}{N}}+|s|^{\frac{N+\alpha}{N-2}}\big),\quad \lim_{s\to 0}F(s)/|s|^{\frac{N+\alpha}{N}}=\lim_{|s|\to+\infty}F(s)/|s|^{\frac{N+\alpha}{N-2}}=0,\quad F(s_0)\neq 0,
\end{equation} 
for some  $s_0\neq 0$ and $C>0$,
see also  the survey \cite{MVSurvey} and the references therein. 
Note that, if $\alpha=0$ in \eqref{eqVM}, since $I_0\ast F(u)=F(u)$,  \eqref{eq:assumVM} covers the Berestycki-Lions growth assumptions only  for the nonnegative (attractive) nonlinearity $F^2(s)\geq 0$  of the corresponding energy functional (see  (3.3) of \cite{BL1}).

On the other hand, as for instance in the Hartree-Fock theory, the interaction potential could be also repulsive \cite{Benguria,LionsCMP}, i.e. with $\beta>0$ and a non-trivial $G(s)\ge 0$.
Moreover problems similar to \eqref{eq} may admit some local terms as well, see also  \cite{MVSurvey} and the references therein.

Our aim is to investigate both nonlinear phenomena with  both nonlocal terms ($0<\beta\leq\alpha$)  in \eqref{eq}, since, in the limiting case $\alpha=\beta=0$, we can fully cover the Berestycki and Lions assumptions \cite{BL1}. 

We impose the following assumptions on $f$ and $g$:
\begin{enumerate}[label=($H_\arabic*$),ref=$H_\arabic*$]
	\item \label{h1} there is a constant $C>0$ and  $p\in \left(\frac{2\b}{N-2}, \frac{N+\b}{N-2}\right]$ such that $|sf(s)|\leq C|s|^{\frac{N+\alpha}{N-2}}$ and $ 0\leq g(s)s\leq C\big(|s|^{p}+|s|^{\frac{N+\beta}{N-2}}\big)$ for $s\in\R$;
	\item \label{ass0} $\displaystyle \lim_{s\to 0}\frac{F(s)}{|s|^\frac{N+\alpha}{N-2}}=\displaystyle \lim_{|s|\to +\infty}\frac{F(s)}{|s|^\frac{N+\alpha}{N-2}}= 0$;
	\item \label{xi1} there is $s_0>0$ such that $F(s_0)\neq 0$; if $\alpha=\beta$, then we assume also $F(s_0)>G(s_0)$.
\end{enumerate}

Observe that, if  $0\le\beta<\frac{N-2}{2}$, then, due to the continuity of $g$, we can take $p=1\in\left(\frac{2\b}{N-2}, \frac{N+\b}{N-2}\right]$.

We remark that these kinds of assumptions follow naturally from the local case, namely when $\alpha=\beta=0$, and equation \eqref{eq} becomes simply 
\begin{equation}\label{eq:BL}
-\Delta u = h(u) \qquad\hbox{ in }\RN.
\end{equation}
This problem has been studied in \cite{BL1} and \cite{Struwe,StruweMA}, under general assumptions. In particular, in \cite{Struwe,StruweMA} Struwe considered  a continuous and odd function $h:~\R\to~\R$ with primitive $H(s)=\int_0^s h(t)\,dt$ such that
\begin{enumerate}[label=(\roman*),ref=\roman*]
	\item \label{intri}$-\infty\leq \limsup_{s\to 0}h(s)s/|s|^{\frac{2N}{N-2}}\leq 0$;
	\item \label{intrii}$-\infty\leq \limsup_{|s|\to +\infty}h(s)s/|s|^{\frac{2N}{N-2}}\leq 0$;
	\item \label{intriii}there is $s_0>0$ such that $ H(s_0)>0$.
\end{enumerate}
Observe that the above assumptions contain those in \cite{BL1}. 
As usual,  by the maximum principle, it is enough to solve \eqref{eq:BL} when $\limsup_{|s|\to +\infty}h(s)s/|s|^{\frac{2N}{N-2}}=0$. 
Now, taking $F$ and $G$ even functions such that $$
F^2(s)=\int_0^s\max\{h(t),0\}\,dt
\quad \hbox{and} \quad 
G^2(s)=\int_0^s\max\{-h(t),0\}\,dt, \qquad \hbox{ for }s\geq 0,
$$
we get $H(s)=F^2(s)-G^2(s)$ and, in the local case $\alpha=\beta=0$, assumptions  \eqref{ass0} and \eqref{xi1} are clearly satisfied. 
Moreover $F$ and $G$ satisfy the following condition
\begin{enumerate}
	\item[$(H_1^*)$]  there is a constant $C>0$ such that $|(F^2)'(s)s|\leq C|s|^{\frac{2N}{N-2}}$ and $ 0\leq (G^2)'(s)s\leq C\big(|s|^{2}+|s|^{\frac{2N}{N-2}}\big)$ for $s\in\R$.
\end{enumerate}
This is a  slightly weaker variant of \eqref{h1}, which is essentially designed  for the nonlocal problem.
In fact, with our argument, one can provide a different proof of the existence of a radial solution under assumptions (\ref{intri})--(\ref{intriii}) from \cite{Struwe,StruweMA}.\\
Further progress on the Berestycki-Lions problem \eqref{eq:BL} has been made in \cite{MederskiBN,MederskiBNZero,JeanjeanLu}; see also the references therein.

We look for a {\em weak solution} $u\in\D$ to \eqref{eq}, i.e.
\begin{equation*}
\int_{\R^N}\nabla u \cdot \n \psi\, dx
=\int_{\R^N} \big(I_\a\ast F(u)\big)f(u)\psi\,dx
-\int_{\R^N} \big(I_\beta\ast G(u)\big)g(u)\psi\,dx
\end{equation*}
for any $\psi\in\cC_0^{\infty}(\R^N)$, where $\D$ stands for the completion of $\cC_0^{\infty}(\R^N)$ with respect to the norm $\|\nabla \cdot\|_2$. 

At least formally solutions of \eqref{eq} are critical points of the functional 
$\cI:\D\to \R\cup \{+\infty\}$ defined as
$$\cI(u)=\int_{\R^N}|\nabla u|^2\, dx-\int_{\R^N} \big(I_\alpha\ast F(u)\big)F(u)\,dx+\int_{\R^N} \big(I_\beta\ast G(u)\big)G(u)\,dx,$$
where $\D$. 
Since $|F(s)|\le C|s|^\frac{N+\alpha}{N-2}$ for some constant $C>0$, we have that $\big(I_\alpha\ast F(u)\big)F(u)\in L^1(\R^N)$. On the other hand $\big(I_\beta\ast G(u)\big)G(u)\in L^1_{\rm loc}(\R^N)$ and need not be integrable in $\R^N$. Therefore $\cI$ may be infinite on a dense subset of $\D$ and, thus, cannot be Fr\'echet-differentiable.

We remark also that scaling properties of the problem play a crucial role, but, in our case, seem to be difficult to apply. Indeed, if $\alpha\neq\beta$, then the nonlinear terms 
\begin{eqnarray*}
	\int_{\R^N}\big(I_\alpha\ast F(u(\lambda\cdot))\big)F(u(\lambda\cdot))\,dx&=&\lambda^{-(N+\alpha)}\int_{\R^N}\big(I_\beta\ast F(u)\big)F(u)\,dx\\
	\int_{\R^N}\big(I_\beta\ast G(u(\lambda\cdot))\big)G(u(\lambda\cdot))\,dx&=&\lambda^{-(N+\beta)}\int_{\R^N}\big(I_\beta\ast G(u)\big)G(u)\,dx
\end{eqnarray*}
have different scaling coefficients and, in particular, one cannot employ Lagrange multipliers as in \cite{BL1},  rescaling as in \cite{StruweMA}, or Pohozaev constraint approach as in \cite{MederskiBN,MederskiBNZero}.

Moreover, to recover the lack of compactness due to the fact that we are working in the whole $\RN$, we start using the invariance of the functional $\cI$ with respect to the orthogonal group action $\mathcal{O}(N)$. Hence we may restrict our considerations to the subspace of radial function $\Dr$, however $\cI|_{\Dr}$ still preserves the above difficulties and may be infinite. 

In this setting, our main result reads as follows.

\begin{Th}\label{ThMain}
	Assume that \eqref{h1}--\eqref{xi1} hold. Then, there is a nontrivial and radial solution $u\in\D$ to \eqref{eq} such that $\big(I_\beta\ast G(u)\big)G(u)\in L^1(\R^N)$.
\end{Th}

Let us describe our variational approach.\\
We observe that 
\begin{equation}\label{fcorsivo}
\mathcal{F}(u):=\int_{\R^N} \big(I_\alpha\ast F(u)\big)F(u)\,dx
\end{equation}
is well-defined on $\D$, however $\cI$ may be infinite. Therefore we replace \begin{equation}\label{defG}
\mathcal{G}(u):=\int_{\R^N} \big(I_\beta * G(u)\big) G(u)\, dx
\end{equation}
with
\begin{equation}\label{defGn}
\mathcal{G}_n(u):=\int_{\R^N} \varphi_n (x)\big(I_\beta * G(u)\big) G(u)\, dx,
\end{equation}
where $\{\vp_n\}_{n\geq 1}$ is a sequence of $\cC_0^{\infty}(\R^N)$ radial functions, decreasing with respect to the radius,  such that, for every $n\geq 1$, $\vp_n(x)=1$ for $x\in B_n$, $\vp_n(x)=0$ 
for $x\in \R^N\setminus B_{2n}$, $0\le \vp_n(x)\leq 1,|x||\nabla\vp_n(x)|\leq c$, 
and $\vp_n(x)\leq \vp_{n+1}(x)$ for $n\geq 1$ and $x\in\R^N$ ($B_n$ stands for the ball of radius $n$ centred at $0$). 
Then $\mathcal{G}_n$ is well-defined on $\D$ and 
\begin{equation}\label{defIn}
\cI_n(u):=\int_{\R^N}|\nabla u|^2\, dx-\mathcal{F}(u)+\mathcal{G}_n(u)
\end{equation}
is of class $\cC^1$. 
\\
The functional $\cI_n$ does not satisfy any variant of Ambrosetti-Rabinowitz condition \cite{AR}, hence it is difficult to find a bounded Palais-Smale sequence on a positive level. Inspired by \cite{HIT,jj} we apply the variational method in \cite[Theorem 2.8]{Willem} to the functional
$$\cJ_n:=(\sigma,u)\in \R\times \Dr\mapsto \cI_n(u(e^{\sigma}\cdot))\in \R.$$ We require a new nonlocal variant of the Brezis-Lieb Lemma for a general nonlinearity, see Lemma \ref{le:BL}, and further compactness properties of $\mathcal{F}(u)$ on $\Dr$ demonstrated in Section \ref{se2}. Then, letting $n\to+\infty$, 
the  careful analysis of the Mountain Pass levels provides a nontrivial radial solution to \eqref{eq}. This approach provides also an alternative proof of the existence of a radial solution in the local case considered in \cite{Struwe,StruweMA}. 
We would like to point out that, contrary to \cite{BL1,Struwe,StruweMA}, we no longer use  the uniform decay at infinity of radial functions from $\Dr$ (see \cite[Radial Lemma A.III]{BL1}) and the compactness lemma due to Strauss \cite[Lemma A.I]{BL1}.

Therefore more can be said in higher dimensions.
Let $N\geq 4$, $N\neq 5$ and similarly as Bartsch and Willem in \cite{BartschWillem} (cf. \cite{MederskiBN,MederskiBNZero,JeanjeanLu}), let us fix $\tau\in\cO(N)$ such that $\tau(x_1,x_2,x_3)=(x_2,x_1,x_3)$ for $x_1,x_2\in\R^M$ and $x_3\in\R^{N-2M}$, where $x=(x_1,x_2,x_3)\in\R^N=\R^M\times\R^M\times \R^{N-2M}$ and $2\leq M\leq N/2$, with $N-2M\neq 1$.
We define
\begin{equation*}
X_{\tau}:=\big\{u\in \D: u(x)=-u(\tau x)\;\hbox{ for all }x\in\R^N\big\}.
\end{equation*}
Clearly, if $u\in X_\tau$ is radial, 
then $u=0$. Hence $X_\tau$ does not contain nontrivial radial functions. Then let us consider $\cO:=\cO(M)\times \cO(M)\times\cO(N-2M)\subset \cO(N)$ acting isometrically on $\D$ with the subspace of invariant function denoted by $\Dnr$.
Moreover our functionals are invariant under this action whenever $f$ and $g$ are odd or even.

Our result, in this setting, is
\begin{Th}\label{ThMain2}
	Assume that \eqref{h1}--\eqref{xi1} hold, $f$ and $g$ are odd, $N\geq 4$ and $N\neq 5$. Then, there is a nontrivial and nonradial solution $u\in \Dnr\cap X_\tau$ to \eqref{eq} such that  $\big(I_\beta\ast G(u)\big)G(u)\in L^1(\R^N)$.
\end{Th}

Observe that in Theorem \ref{ThMain} and Theorem \ref{ThMain2} we can take $G(s)=s$ and $\beta=0$ and we obtain solutions in $H^1(\R^N)$ solving the Choquard problem \eqref{eqVM}. 
In fact, dealing with the operator $-\Delta u+ u$, more general assumptions imposed on $F$ can be considered, which fully cover situation in \cite{MVTrans}. 

Actually, our argument can be, quite easily, adapted to the following problem
\begin{equation}\label{eq2}
-\Delta u + \omega u= \big(I_\alpha\ast F(u)\big)f(u)-\big(I_\beta\ast G(u)\big)g(u)
\qquad \hbox{ in }\R^N,
\end{equation}
where $\omega>0 $, assuming that
\begin{enumerate}[label=($H_\arabic*'$),ref=$H_\arabic*'$]
	\item \label{h1'} there is a constant $C>0$ and  $p\in \left(\frac{2\b}{N-2}, \frac{N+\b}{N-2}\right]$ such that $|sf(s)|\leq C(|s|^{\frac{N+\alpha}{N}}+|s|^{\frac{N+\alpha}{N-2}})$ and $ 0\leq g(s)s\leq C\big(|s|^{p}+|s|^{\frac{N+\beta}{N-2}}\big)$ for $s\in\R$;
	\item \label{ass0'} $\displaystyle \lim_{s\to 0}\frac{F(s)}{|s|^\frac{N+\alpha}{N}}=\displaystyle \lim_{|s|\to +\infty}\frac{F(s)}{|s|^\frac{N+\alpha}{N-2}}= 0$;
	\item \label{xi1'} there is $s_0>0$ such that $F(s_0)\neq 0$; we assume also $F(s_0)>G(s_0)$, if $\alpha=\beta>0$, and $F^2(s_0)>G^2(s_0)+\omega s_0^2$, if $\alpha=\beta=0$.
\end{enumerate}
Observe that the energy functional associated with \eqref{eq2} is given by $$\mathcal{K}_\omega(u):=\cI(u)+ \omega \int_{\R^N}|u|^2\, dx,\quad u\in H^1(\R^N),$$
and may be also infinite due to the possible nonintegrable term $\big(I_\beta\ast G(u)\big)G(u)$.

Our results for equation \eqref{eq2} read as follows.

\begin{Th}\label{ThMain3}
Assume that  \eqref{h1'}--\eqref{xi1'} hold. Then, there is a nontrivial and radial solution $u$ to \eqref{eq2} in $H^1(\R^N)$ such that  $\big(I_\beta\ast G(u)\big)G(u)\in L^1(\R^N)$. Moreover, if  $f$ and $g$ are odd, $N\geq 4$ and $N\neq 5$, there is also a nontrivial and nonradial solution $v$ to \eqref{eq2} in $H^1(\R^N)\cap X_\tau$ such that  $\big(I_\beta\ast G(v)\big)G(v)\in L^1(\R^N)$.
\end{Th}

In particular, if 
\begin{equation}\label{eq:CubicQuintic}
		F(s):=\frac{1}{\sqrt{q}} |s|^{q} \hbox{ with } 1<q<\frac{N+\alpha}{N-2},\hbox{ and }G(s):=\sqrt{\frac{N-2}{N+\beta}} |s|^{\frac{N+\beta}{N-2}},
\end{equation}
then $F$ and $G$
satisfy  \eqref{h1'}--\eqref{xi1'}  if and only if $\omega\in (0,\omega_0)$, where
	\begin{equation*}
\omega_0:=
	\begin{cases}
\frac{2^*-2q}{2^*(q-1)}\Big(\frac{N(q-1)}{2q}\Big)^{\frac{2^*-2}{2^*-2q}}
&\hbox{ if }\alpha=\beta=0,\\
+\infty &\hbox{ if }\alpha>0.
\end{cases}
\end{equation*}
Then, finally, we obtain the following corollary.
\begin{Cor}\label{th:cubicquintic}
	Suppose that $F$ and $G$ are given by \eqref{eq:CubicQuintic}.
	\begin{enumerate} [label=(\alph*),ref=\alph*]
		\item \label{a14} For any $\omega\in (0,\omega_0)$ there is a radially symmetric symmetric solution in $H^1(\R^N)$ and a nonradial solution in  $H^1(\R^N)\cap X_\tau$ to \eqref{eq2}.
		\item \label{b14} If $\omega\notin (0,\omega_0)$, then \eqref{eq2} has only trivial finite energy solution.
	\end{enumerate}
\end{Cor}

Corollary \ref{th:cubicquintic} has been known only in the local case $\alpha=\beta=0$ and the problem appears in nonlinear optics as well as in the the study of Bose–Einstein condensates \cite{Gammal,Mihalache}.
Note that solutions exist only for $0<\omega<\omega_0<+\infty$, see e.g. \cite{MederskiBNZero,Killip,BL1}. In the nonlocal case, for instance if $N=3$, $q=2$ and $\alpha>\beta=0$, 
we solve the nonlocal cubic-quintic problem of the nonlinear optics for all $\omega>0$, where $I_\alpha$ is a nonlocal response function determined by the details of the physical process responsible for the nonlocality \cite{Chenetal}.

Through the paper we use the following notation.
\\
We denote by $\|\cdot\|_k$ the usual norm in $L^k(\R^N)$, for $k\geq 1$, and by  $B_R$ the ball centered in $0$ with radius $R>0$ in $\R^N$.
Recall that $2^*=\frac{2N}{N-2}$. Finally $C$ is a generic positive constant which may vary from line to line.

\section{Functional setting and compactness properties}\label{se2}

We prove our results for $\beta>0$, the most difficult and fully nonlocal situation. Thus, from now on, we assume that $0<\beta\le\alpha<N$ and \eqref{h1}--\eqref{xi1} hold, with $p=1$ when $0<\beta<\frac{N-2}{2}$. The proofs of the paper are simplified when $\beta=0$ or $\alpha=\beta=0$ and we skip these cases.



It is standard to see that the functional $\mathcal{F}:L^{2^*}(\R^N)\to \R$, defined in \eqref{fcorsivo}
is of class $\cC^1$, cf.  \cite{MVTrans}.

In order to control the convergence of $\mathcal{F}$, we need the following nonlocal variant of the Brezis-Lieb Lemma \cite{Brezis} for the general nonlinarity.
Note that nonlocal variants for particular nonlinearities have already appeared in \cite[Lemma 2.2]{Bellazzini}, \cite[Lemma 2.4]{MVJFA}. The proofs of \cite{Bellazzini,MVJFA} seem to be difficult to adapt to the general nonlinear term. We provide an independent proof for any continuous $f$ satisfying \eqref{h1} and  \eqref{ass0}.

\begin{Lem}\label{le:BL}
Let  $u_n\weakto u_0$  in $\D$. Then
		\begin{equation*}
		\begin{split}
		&\lim_{n} \left(\int_{\R^N} \big(I_\a\ast F(u_n\big)f(u_n)u_n\,dx-\int_{\R^N} \big(I_\a\ast F(u_n-u_0)\big)f(u_n)u_n\,dx\right)\\
		&\qquad
		=\int_{\R^N} \big(I_\a\ast F(u_0)\big)f(u_0)u_0\,dx.
		\end{split}
		\end{equation*}
\end{Lem}
\begin{proof}
 We claim that, passing to a subsequence, for any $s\in [0,1]$, 
	\begin{equation}\label{BLeq1}
	\lim_{n}\irn \big(I_\a\ast (f(u_n)u_n)\big)f(u_n-su_0) u_0\, dx = 
	\irn \big(I_\a\ast (f(u_0)u_0)\big)f(u_0-su_0)u_0\, dx.
	\end{equation}
Let $\eps>0$ and $\psi\in \cC^\infty_0(\RN)$ such that $\|u_0-\psi\|_{2^*}<\eps$. We have
\begin{align*}
&\left|\irn \big(I_\a\ast (f(u_n)u_n)\big)f(u_n-su_0) u_0\, dx -
\irn \big(I_\a\ast (f(u_0)u_0)\big)f(u_0-su_0)u_0\, dx\right|
\\
&\qquad\le\underbrace{ \left|\irn \big(I_\a\ast (f(u_n)u_n)\big)f(u_n-su_0)( u_0 -\psi)\, dx\right|}_{(A)}
\\
&\qquad \quad +\underbrace{\left|\irn \big(I_\a\ast (f(u_n)u_n)\big)\big(f(u_n-su_0)-f(u_0-su_0)\big)\psi\, dx\right|}_{(B)}
\\
&\qquad \quad +\underbrace{\left|\irn \Big(\big(I_\a\ast (f(u_n)u_n)\big)- \big(I_\a\ast (f(u_0)u_0)\big)\Big)f(u_0-su_0)\psi\, dx\right|}_{(C)}
\\
&\qquad \quad +\underbrace{\left|
	\irn \big(I_\a\ast (f(u_0)u_0)\big)f(u_0-su_0)(\psi-u_0)\, dx\right|}_{(D)}.
\end{align*} 
Since $\{u_n\}$ is a bounded sequence in $L^{2^*}(\RN)$, we deduce by \eqref{h1} that $\{f(u_n)u_n\}$ is bounded in $L^{\frac{2N}{N+\a}}(\RN)$. Moreover, by the continuity, we deduce that $f(u_n)u_n$ converges to $f(u_0)u_0$ a.e. on $\R^N$ along a subsequence. Therefore $f(u_n)u_n$ tends weakly to $f(u_0)u_0$ in $L^{\frac{2N}{N+\a}}(\RN)$. 
As the Riesz potential defines a linear and continuous map from
$L^{\frac{2N}{N+\a}}(\RN)$ to $L^{\frac{2N}{N-\a}}(\RN)$, we obtain that
$I_\a\ast (f(u_n)u_n)$ tends weakly to $I_\a\ast( f(u_0)u_0)$ in $L^{\frac{2N}{N-\a}}(\RN)$.
Moreover, since $u_n-su_0$ converges to $u_0-su_0$ in $L^q_{\rm loc}(\RN)$, for $1\le q<2^*$, by \eqref{h1} we infer that $f(u_n-su_0)$ converges to $f(u_0-su_0)$ in $L^q_{\rm loc}(\RN)$, for $1\le q<2N/(\a+2)$.\\
Then, by the Hardy–Littlewood–Sobolev inequality and since  $\{f(u_n-su_0)\}$ is bounded in $L^\frac{2N}{\alpha+2}(\R^N)$ we obtain
\[
(A)\le 
C\|f(u_n)u_n\|_{\frac{2N}{N+\a}} \|f(u_n-su_0)\|_{\frac{2N}{\a+2}} \| u_0 -\psi\|_{2^*}
\le C \eps
\]
and analogously, $(D)\le C\eps$.\\
Moreover, denoting by $K:=\supp (\psi)$, we have
\[
(B)
\leq
C \|f(u_n)u_n\|_\frac{2N}{N+\alpha}\|f(u_n-su_0)-f(u_0-su_0)\|_{L^\frac{N(N+2\alpha+2)}{(N+\alpha)(\alpha+2)}(K)}
\|\psi\|_\frac{2N(N+2\alpha+2)}{(N+\a)(N-2)}
=o_n(1).
\]
Finally, also $(C)=o_n(1)$, since $f(u_0-su_0)\psi$ belongs to $ L^{\frac{2N}{N+\a}}(\RN)$, namely the dual space of $L^{\frac{2N}{N-\a}}(\RN)$.\\
Therefore \eqref{BLeq1} is proved.\\
Now, for any $n\in \N$,  we  set $\phi_n(s)=\big(I_\a\ast F(u_n-su_0)\big)f(u_n)u_n$, for  $s\in [0,1]$, and we obtain
\begin{align*}
&\int_{\R^N} \big(I_\a\ast F(u_n\big)f(u_n)u_n\,dx-\int_{\R^N} \big(I_\a\ast F(u_n-u_0)\big)f(u_n)u_n\,dx\\
&\qquad
=\irn \big(\phi_n(0)-\phi_n(1)\big)dx
=-\int_0^1 \left(\irn \phi_n'(s)\,dx\right) ds
\\
&\qquad
=\int_0^1 \left(\irn \big(I_\a\ast (f(u_n-su_0)u_0)\big)f(u_n)u_n \,dx\right) ds
\\
&\qquad
=\int_0^1 \left(\irn \big(I_\a\ast (f(u_n)u_n)\big)f(u_n-su_0)u_0 \,dx\right) ds.
\end{align*}
Hence, by \eqref{BLeq1}, taking into account the Lebesgue Dominated Convergence Theorem
\begin{align*}
&\lim_{n} \left(\int_{\R^N} \big(I_\a\ast F(u_n\big)f(u_n)u_n\,dx-\int_{\R^N} \big(I_\a\ast F(u_n-u_0)\big)f(u_n)u_n\,dx\right)
\\
&\qquad=\lim_{n}\int_0^1 \left(\irn \big(I_\a\ast (f(u_n)u_n) \big)f(u_n-su_0)u_0\,dx\right) ds
\\
&\qquad=\int_0^1 \left(\lim_{n}\irn \big(I_\a\ast (f(u_n)u_n) \big)f(u_n-su_0)u_0\,dx\right) ds
\\
&\qquad=\int_0^1 \left(\irn \big(I_\a\ast (f(u_0)u_0)\big)f(u_0-su_0)u_0\, dx\right) ds
\\
&\qquad =-\int_0^1 \left(\irn \phi_0'(s)\,dx\right) ds
=-\irn \left(\int_0^1 \phi_0'(s)\, ds\right)dx
\\
&\qquad =\irn \big(\phi_0(0)-\phi_0(1)\big)dx
=\int_{\R^N} \big(I_\a\ast F(u_0)\big)f(u_0)u_0\,dx.
\end{align*}
\end{proof}

Now, let $\cO'=\cO(N)$, or $\cO'=\cO=\cO(M)\times \cO(M)\times\cO(N-2M)\subset \cO(N)$ provided that $N\geq 4$ and $N\neq 5$ with $2\leq M\leq N/2$ and $N-2M\neq 1$.
Let $\DO$ be the subspace of $\cO'$-invariant functions. Below we demonstrate the compactness properties in the following lemmas.

\begin{Lem}\label{le:BLnew}
	Let $u_n\weakto u_0$  in $\DO$. Then
\begin{equation*}
	\lim_{n} \int_{\R^N} \big(I_\a\ast F(u_n)\big)f(u_n)u_n\,dx
	=\int_{\R^N} \big(I_\a\ast F(u_0)\big)f(u_0)u_0\,dx.
\end{equation*}
\end{Lem}
\begin{proof}
By Lemma \ref{le:BL}, we conclude if we prove that
	\begin{equation*}
	\lim_{n} \int_{\R^N} \big(I_\a\ast F(u_n-u_0)\big)f(u_n)u_n\,dx
	=0.
	\end{equation*}
Indeed, by the Hardy-Littlewood-Sobolev inequality and \eqref{h1}, we have
	\[
	\irn (I_\alpha*F(u_n-u_0))f(u_n)u_n \, dx
	\leq
	C \| F(u_n-u_0) \|_{\frac{2N}{N+\alpha}}\| f(u_n)u_n \|_{\frac{2N}{N+\alpha}}
	\le	C \| F(u_n-u_0) \|_{\frac{2N}{N+\alpha}}.
	\]
	The fact that $\| F(u_n-u_0) \|_{2N/(N+\alpha)}\to 0$ is a consequence of \eqref{ass0} and \cite[Lemma A.1]{MederskiBNZero}, where the symmetry $\cO'$ plays a crucial r\^ole.
\end{proof}

\begin{Lem}\label{le:BL1}
	Let $u_n\weakto u_0$ in $\D$. Then, for any $\psi\in \cC_0^\infty(\RN)$,
	\begin{equation}\label{cacchio}
	\lim_{n} \int_{\R^N} \big(I_\a\ast F(u_n)\big)f(u_n)\psi\,dx
	=\int_{\R^N} \big(I_\a\ast F(u_0)\big)f(u_0)\psi\,dx.
	\end{equation}	
\end{Lem}
\begin{proof}
Arguing as in the proof of Lemma \ref{le:BL} and passing to a subsequence, we have that $f(u_n)\to f(u_0)$ in $L^q_{\rm loc}(\RN)$, for $1\le q<2N/(\a+2)$, 
and 
$\{I_\a\ast F(u_n)\}$ is bounded in $L^{\frac{2N}{N-\a}}(\RN)$
and
tends weakly to $I_\a\ast F(u_0)$ in $L^{\frac{2N}{N-\a}}(\RN)$.   Thus, since
\begin{multline*}
\left|
\int_{\R^N} \big(I_\a\ast F(u_n)\big)f(u_n)\psi\,dx
-\int_{\R^N} \big(I_\a\ast F(u_0)\big)f(u_0)\psi\,dx
\right|
\\
\le \int_{\R^N} \big(I_\a\ast F(u_n)\big)|f(u_n)-f(u_0)||\psi|\,dx
+
\left|\int_{\R^N} \big( I_\a\ast F(u_n)-I_\a\ast F(u_0)\big) f(u_0)\psi\,dx\right|
,
\end{multline*}
using the same arguments as in (B) and (C) in the proof of Lemma \ref{le:BL}, we get
\eqref{cacchio}.
\end{proof}

For what concerns the term with $G$,
at least formally, we define the functional $\mathcal{G}$ as in \eqref{defG}.
However, if in \eqref{h1},  
 $p<\frac{N+\b}{N-2}$, the situation is quite different from $\mathcal{F}$ and $\mathcal{G}$ need not be finite. Indeed, in such a case, let us consider the Banach spaces
\[
L^\mu(\Omega)+L^\nu(\Omega)
:=
\left\{
u\in\mathcal{M}(\Omega): u=u_1+u_2, u_1\in L^\mu(\Omega), u_2\in L^\nu(\Omega)
\right\},
\]
where $1\leq\mu\leq\nu<+\infty$, $\Omega$ is an arbitrary subset of $\R^N$, and $\mathcal{M}(\Omega)$ is the set of the real measurable functions defined on $\Omega$, equipped with the norm
\[
\|u\|_{\mu,\nu}:=\inf_{u=u_1+u_2} (\|u_1\|_{L^\mu(\Omega)}+\|u_2\|_{L^\nu(\Omega)})
\]
(see e.g. \cite[Section 2]{BPR} for more details about these spaces).
\\
Observe that if $u\in \D\subset L^{2^*}(\R^N)$, since $|u|^p\in L^{\frac{2^*}{p}}(\RN)$ and   $|u|^\frac{N+\beta}{N-2}\in L^{\frac{2N}{N+\beta}}(\R^N)$, by \cite[Proposition 2.3]{BPR} and \eqref{h1}, we get
\begin{equation}\label{Gu}
G(u)\in L^{\frac{2N}{N+\beta}}(\R^N)+
L^{\frac{2^*}p}(\R^N). 
\end{equation}
Moreover, since
\[
I_\beta * G(u)\leq C \big(I_\beta *(|u|^p+|u|^\frac{N+\beta}{N-2})\big),
\]
by  \cite[Inequality (9), page 107]{LL} and \cite[Proposition 2.3]{BPR}
we have
\begin{equation}\label{Gu2}
I_\beta * G(u)\in L^{\frac{2N}{N-\beta}}(\R^N)+L^{\frac{2N}{(N-2)p-2\beta}}(\R^N).
\end{equation}
However, this does not seem enough to assure that $\mathcal{G}(u)<+\infty$ for any $u\in \D$,  and so we need a different approach. We replace $
\mathcal{G}(u)$ with
$\mathcal{G}_n(u)$ given by \eqref{defGn} together with the sequence $\{\varphi_n\}$ defined there.

%

We prove the following lemma.
\begin{Lem}
For every $n\in\mathbb{N}$, $\mathcal{G}_n\in \cC^1(\D,\R)$.
\end{Lem}
\begin{proof}
We divide the proof in five steps.\\
{\bf Step 1:} {\em $\mathcal{G}_n$ is well defined.}\\
Observe that 
\[
0 \leq \mathcal{G}_n(u)\leq \int_{B_{2n}} \big(I_\beta * G(u)\big) G(u)\, dx
\]
and, by \eqref{Gu} and \eqref{Gu2}, 
$I_\beta * G(u) \in L^\frac{2N}{N-\beta}(B_{2n})+ L^\frac{2N}{(N-2)p-2\beta}(B_{2n})\subset L^\frac{2N}{N-\beta}(B_{2n})$ and $G(u) \in L^{\frac{2N}{N+\beta}}(B_{2n})+L^{\frac{2^*}p}(B_{2n})\subset L^{\frac{2N}{N+\beta}}(B_{2n})$. Thus, the H\"older inequality allows us to conclude.\\
{\bf Step 2:} 
{\em if $\{u_m\}\subset L^{2^*}(\R^N)$ and $u_m\to u$ in $L^{2^*}(\R^N)$, then, up to a subsequence, $I_\beta * G(u_m) \to I_\beta * G(u)$ a.e. in $\R^N$, as $m\to +\infty$.}
\\
Since  $u_m \to u$ in $L^{2^*}(\R^N)$, then, up to a subsequence, there exist $\Omega_1\subset\RN$ with $|\Omega_1|=0$ and $w\in L^{2^*}(\R^N)$ such that $|u_m|\leq w$ and $u_m \to u$ in $\R^N\setminus \Omega_1$. 
\\
Since $w^p + w^\frac{N+\beta}{N-2}\in L^{\frac{2N}{N+\beta}}(\R^N)+L^{\frac{2^*}p}(\R^N)$,  by  \cite[Inequality (9), page 107]{LL}, we have that $ I_\beta * \left(w^p + w^\frac{N+\beta}{N-2}\right) \in L^{\frac{2N}{N-\beta}}(\R^N)+L^{\frac{2N}{(N-2)p-2\beta}}(\R^N)$ and so, there exists $\Omega_2\subset \RN$, with $|\Omega_2|=0$,  such that
\[
\frac{w^p(y) + w^\frac{N+\beta}{N-2}(y)}{|x-y|^{N-\beta}} \in L^1(\R^N),
\quad
\hbox{ for all }x\in \R^N\setminus \Omega_2.
\]
Thus, if we fix $x\in\R^N\setminus \Omega_2$, we have that
\[
\frac{G(u_m(y))}{|x-y|^{N-\beta}} \to \frac{G(u(y))}{|x-y|^{N-\beta}}, \quad
\text{ for all }y\in \R^N\setminus \Omega_1
\]
and
\[
\frac{G(u_m(y))}{|x-y|^{N-\beta}}
\leq
C \frac{|u_m(y)|^p + |u_m(y)|^\frac{N+\beta}{N-2}}{|x-y|^{N-\beta}}
\leq
C \frac{w^p(y) + w^\frac{N+\beta}{N-2}(y)}{|x-y|^{N-\beta}} \in L^1(\R^N).
\]
Hence, by the Lebesgue Dominated Convergence Theorem we can conclude.
\\
{\bf Step 3:} {\em $\mathcal{G}_n$ is continuous.}
\\
Let $\{u_m\} \subset \D$ be such that $u_m \to u$ in $\D$ as $m \to +\infty$. Up to a subsequence we have that $u_m \to u$ in $L^{2^*}(\R^N)$, $u_m \to u$ a.e. in $\R^N$, and there exists $w \in L^{2^*}(\R^N)$ such that $|u_m| \leq w$ a.e. in $\R^N$. Thus, since $G$ is continuous, $G(u_m) \to G(u)$ a.e. in $\R^N$ and, by Step 2, $I_\beta * G(u_m) \to I_\beta * G(u)$ a.e. in $\R^N$. Hence
\[
\varphi_n (x)\big(I_\beta * G(u_m)\big) G(u_m) \to \varphi_n (x)\big(I_\beta * G(u)\big) G(u)
\text{ a.e. in } \R^N, \text{ as }m\to+\infty.
\]
Moreover,
\[
0\leq
\varphi_n (x)\big(I_\beta * G(u_m)\big) G(u_m)
\leq
C \varphi_n (x)\big(I_\beta *(w^p + w^\frac{N+\beta}{N-2})\big)(w^p + w^\frac{N+\beta}{N-2})
\in L^1(\R^N)
\]
since, arguing as before,  $I_\beta * (w^p + w^\frac{N+\beta}{N-2}) \in  L^{\frac{2N}{N-\b}}(B_{2n})$ and $w^p + w^\frac{N+\beta}{N-2} \in L^{\frac{2N}{N+\beta}}(B_{2n})$. Thus, the Lebesgue Dominated Convergence Theorem allows us to conclude.\\
{\bf Step 4:} {\em $\mathcal{G}_n$ is differentiable and, for any $v\in \D$,
\[
\mathcal{G}_n'(u)[v]=2 \int_{\R^N} \varphi_n (x)\big(I_\beta * G(u)\big) g(u)v\, dx.
\]}
First we prove that
\begin{equation*}
\left| \int_{\R^N} \varphi_n (x)\big(I_\beta * G(u)\big) g(u)v\, dx \right| < +\infty.
\end{equation*}
Observe that
\[
\left| \int_{\R^N} \varphi_n (x)\big(I_\beta * G(u)\big) g(u)v\, dx \right|
\leq
\int_{B_{2n}} \big(I_\beta * G(u)\big) |g(u)| |v|\, dx,
\]
and, by assumptions on $g$,
\begin{equation}\label{gu}
|g(u)|\leq C(|u|^{p-1} + |u|^\frac{\beta+2}{N-2})
\in \begin{cases}
L^\frac{2N}{\beta+2}(\R^N)+L^\infty(\R^N), & \hbox{if } 0<\b<\frac{N-2}2,
\\
L^\frac{2N}{\beta+2}(\R^N)+L^\frac{2^*}{p-1}(\R^N),
 & \hbox{if } \frac{N-2}2\le\b<N.
\end{cases}
\end{equation}
In any case we have that $I_\beta * G(u) \in  L^\frac{2N}{N-\b}(B_{2n})$ and, by \eqref{gu}, $g(u)\in L^\frac{2N}{\beta+2}(B_{2n})$. Thus, the H\"older inequality allows us to conclude.\\
Finally, arguing as before, we prove that the map $$v\in \D\longmapsto  \int_{\R^N} \varphi_n(x) \big(I_\beta * G(u)\big) g(u)v\, dx$$ is continuous and this implies also the claim.\\
{\bf Step 5:} {\em $\mathcal{G}_n'$ is continuous.}
\\
Let $v\in \D$, with $\|\n  v\|_2 \leq 1$ and  $\{u_m\} \subset \D$ be such that $u_m \to u$ in $\D$ as $m \to +\infty$. Up to a subsequence we have that $u_m \to u$ in $L^{2^*}(\R^N)$, $u_m \to u$ a.e. in $\R^N$, and there exists $w \in L^{2^*}(\R^N)$ such that $|u_m| \leq w$ a.e. in $\R^N$. 
Moreover
\begin{align*}
\left| \mathcal{G}_n' (u_m)[v] - \mathcal{G}_n' (u)[v] \right|
& \leq
\int_{B_{2n}} \Big|\big(I_\beta * G(u_m)\big)g(u_m) -\big(I_\beta * G(u)\big)g(u)\Big| |v|\, dx
\\
& \leq 
C\left( \int_{B_{2n}} \Big|\big(I_\beta * G(u_m)\big)g(u_m) -\big(I_\beta * G(u)\big)g(u)\Big|^\frac{2N}{N+2}\, dx \right)^\frac{N+2}{2N} .
\end{align*}
Using also Step 2, we have that $\big(I_\beta * G(u_m)\big)g(u_m) \to [I_\beta * G(u)]g(u)$ a.e. in $\R^N$, and so, observing that, by \eqref{h1}, 
\begin{align*}
0\leq I_\beta*G(u_m) &\leq C I_\beta * (w^{p}+w^\frac{N+\beta}{N-2})\in L^\frac{2N}{N-\beta}(B_{2n}),\\
0\leq I_\beta*G(u) &\leq C I_\beta * (|u|^{p}+|u|^\frac{N+\beta}{N-2})\in L^\frac{2N}{N-\beta}(B_{2n}),
\end{align*}
and 
\begin{align*}
|g(u_m)| &\leq C (w^{p-1}+w^\frac{\beta+2}{N-2})\in L^\frac{2N}{\beta + 2}(B_{2n}),\\
|g(u)| &\leq C (|u|^{p-1}+|u|^\frac{\beta+2}{N-2})\in L^\frac{2N}{\beta + 2}(B_{2n}),
\end{align*}
we can conclude by the Lebesgue Dominated Convergence Theorem.
\end{proof}

Now we prove this further compactness result.
\begin{Lem}\label{pippa}
Let $u_n\weakto u_0$ in $\D$.
Then, for any $\psi\in\cC_0^{\infty}(\R^N)$,
\[
\lim_{n}\int_{\R^N} \varphi_n (x)(I_\beta*G(u_n)) g(u_n) \psi \, dx=
\int_{\R^N}  (I_\beta*G(u_0)) g(u_0) \psi\, dx.
\]
\end{Lem}
\begin{proof}
Of course it is enough to show that
\[
\lim_{n}\int_{\operatorname{supp}(\psi)}  \big(I_\beta*G(u_n)\big) g(u_n) \psi \, dx= \int_{\operatorname{supp}(\psi)}  \big(I_\beta*G(u_0)\big) g(u_0) \psi\, dx,
\]
recalling that $\operatorname{supp}(\psi)$ is compact and then, for $n$ large enough, $\operatorname{supp}(\psi)\subset B_{n}$.\\
Since $u_n \rightharpoonup  u_0$ weakly in $\D$,  up to a subsequence, $u_n \to u_0$ a.e. in $\R^N$ and so $G(u_n) \to G(u_0)$ a.e. in $\R^N$, as $n\to +\infty$.\\
Moreover $\{G(u_n)\}$ is bounded in $L^{\frac{2N}{N+\beta}}(\R^N)+L^{\frac{2^*}p}(\R^N)$. Indeed, by the assumptions on $g$, the definition of the norm in $L^{\frac{2N}{N+\beta}}(\R^N)+L^{\frac{2^*}p}(\R^N)$, and \cite[Corollary 2.12]{BPR}, we have
\[
\|G(u_n)\|_{\frac{2N}{N+\beta},\frac{2^*}p}
\leq
C (\|u_n\|_{2^*}^p+\|u_n\|_{2^*}^\frac{N+\beta}{N-2})\leq C.
\]
Thus, the reflexivity of $L^{\frac{2N}{N+\beta}}(\R^N)+L^{\frac{2^*}p}(\R^N)$ (see \cite[Corollary 2.11]{BPR}) implies that there exists $\tilde{u}\in L^{\frac{2N}{N+\beta}}(\R^N)+L^{\frac{2^*}p}(\R^N)$ such that, up to a subsequence,  $G(u_n)\rightharpoonup \tilde{u}$ in $L^{\frac{2N}{N+\beta}}(\R^N)+L^{\frac{2^*}p}(\R^N)$.\\
We claim that $\tilde{u}=G(u_0)$.\\
Indeed, using a classical argument,
(see e.g. \cite[Exercise 3.4]{Brezis}),
the weak convergence $G(u_n)\rightharpoonup \tilde{u}$ in $L^{\frac{2N}{N+\beta}}(\R^N)+L^{\frac{2^*}p}(\R^N)$ implies that there exists a sequence $\{z_n\}\subset L^{\frac{2N}{N+\beta}}(\R^N)+L^{\frac{2^*}p}(\R^N)$ such that, for all $n\in\N$,
\[
z_n \in \operatorname{conv}\Big(\bigcup_{i=1}^n \{G(u_i)\}\Big)
\]
and $z_n \to \tilde{u}$ in $L^{\frac{2N}{N+\beta}}(\R^N)+L^{\frac{2^*}p}(\R^N)$.  Thus, by \cite[Proposition 2.8]{BPR}, up to a subsequence, we get that $z_n \to \tilde{u}$ a.e. in $\R^N$, that allows us to conclude.\\
About the sequence $\{I_\beta *G(u_n)\}$, since by \eqref{h1}
\[
0\le I_\beta * G(u_n) \leq C(I_\beta * |u_n|^p + I_\beta * |u_n|^\frac{N+\beta}{N-2}) \in L^{\frac{2N}{N-\beta}}(\R^N)+L^{\frac{2N}{(N-2)p-2\beta}}(\R^N),
\]
using \cite[Corollary 2.12]{BPR}, we have
\begin{align*}
\|I_\beta *G(u_n)\|_{\frac{2N}{N-\beta},\frac{2N}{(N-2)p-2\beta}}
&\leq C (\| I_\beta * |u_n|^\frac{N+\beta}{N-2}\|_\frac{2N}{N-\beta}
+ \| I_\beta * |u_n|^p\|_\frac{2N}{(N-2)p-2\beta})\\
&\leq C (\| u_n\|_{2^*}^\frac{N+\beta}{N-2}
+ \| u_n\|_{2^*}^p)
\leq C.
\end{align*}
Moreover, observe that the linear functional
\[
w\in L^\frac{2N}{N+\beta}(\R^N)+L^\frac{2^*}{p}(\R^N) \mapsto I_\beta *w \in L^{\frac{2N}{N-\beta}}(\R^N)+L^{\frac{2N}{(N-2)p-2\beta}}(\R^N)
\]
is continuous. Indeed, if $w\in  L^\frac{2N}{N+\beta}(\R^N)+L^\frac{2^*}{p}(\R^N)$, $w=w_1+w_2$ with $w_1\in L^\frac{2N}{N+\beta}(\R^N) $ and $w_2\in L^\frac{2^*}{p}(\R^N)$, by  \cite[Inequality (9), page 107]{LL} we get
	\[
	\|I_\beta *w \|_{\frac{2N}{N-\beta},\frac{2N}{(N-2)p-2\beta}}
	\leq
	\|I_\beta *w_1 \|_{\frac{2N}{N-\beta}} + \|I_\beta *w_2 \|_{\frac{2N}{(N-2)p-2\beta}}
	\leq
	C(  \| w_1 \|_\frac{2N}{N+\beta} + \| w_2 \|_\frac{2^*}{p})
	\]
	and, passing to the infimum on $w_1\in L^\frac{2N}{N+\beta}(\R^N) $ and $w_2\in L^\frac{2^*}{p}(\R^N)$, we conclude.\\
This, combined with the weak convergence $G(u_n)\rightharpoonup G(u_0)$ in $L^{\frac{2N}{N+\beta}}(\R^N)+L^{\frac{2^*}p}(\R^N)$, implies that $I_\beta * G(u_n)\rightharpoonup I_\beta*G(u_0)$ in $L^{\frac{2N}{N-\beta}}(\R^N)+L^{\frac{2N}{(N-2)p-2\beta}}(\R^N)$.\\
Hence, as done for $f$ in Lemma \ref{le:BL1}, we have that
\begin{multline*}
\left|
\int_{\operatorname{supp}(\psi)} \big(I_\b\ast G(u_n)\big)g(u_n)\psi\,dx
-\int_{\operatorname{supp}(\psi)} \big(I_\b\ast G(u_0)\big)g(u_0)\psi\,dx
\right|
\\
\le \int_{\operatorname{supp}(\psi)} \big(I_\b\ast G(u_n)\big)|g(u_n)-g(u_0)||\psi|\,dx
+
\left|\int_{\R^N} \big( I_\b\ast G(u_n)-I_\a\ast G(u_0)\big) g(u_0)\psi\,dx\right|.
\end{multline*}
About the first integral, observe that, the boundedness of $\{u_n\}$ in $\D$ implies also that $u_n \to u_0$ in $L_{\rm loc}^\tau (\R^N)$, for all $1\leq \tau<2^*$ and  so, for any fixed $1\leq \tau<2^*$ and $K\subset\subset \R^N$,
up to a subsequence,  there exists $w_K\in L^\tau (K)$ such that $|u_n|\leq w_K$ a.e. in $K$.
Thus,  denoting for simplicity $w:=w_{\operatorname{supp}(\psi)}$ and taking for instance
\[
\tau=\frac{N(N+2\beta+2)}{(N-2)(N+\beta)},
\]
by the assumptions on $g$ we have
\[
g(u_n)\to g(u_0)
\text{ a.e. in } \operatorname{supp}(\psi),
\]
\[
|g(u_n)|
\leq C (|u_n|^{p-1} + |u_n|^\frac{\beta+2}{N-2})
\leq C (w^{p-1} + w^\frac{\beta+2}{N-2})
\in L^\frac{N(N+2\beta+2)}{(N+\beta)(\beta+2)}(\operatorname{supp}(\psi)),
\]
\[
|g(u_0)|
\leq C (|u_0|^{p-1} + |u_0|^\frac{\beta+2}{N-2})
 \in L^\frac{N(N+2\beta+2)}{(N+\beta)(\beta+2)}(\operatorname{supp}(\psi)).
\]
Moreover, the boundedness of $\{I_\beta *G(u_n)\}$ in $L^{\frac{2N}{N-\beta}}(\R^N)+L^{\frac{2N}{(N-2)p-2\beta}}(\R^N)$ implies its boundedness in $L^{\frac{2N}{N-\beta}}(\operatorname{supp}(\psi))+L^{\frac{2N}{(N-2)p-2\beta}}(\operatorname{supp}(\psi))= L^{\frac{2N}{N-\beta}}(\operatorname{supp}(\psi))$.\\
Thus, by the H\"older inequality and the Lebesgue Dominated Convergence Theorem, we have
\begin{align*}
\int_{\operatorname{supp}(\psi)} \big(I_\b\ast G(u_n)\big)|g(u_n)-g(u_0)||\psi|\,dx
&\leq
C \Big(\int_{\operatorname{supp}(\psi)} |g(u_n)-g(u_0)|^\frac{2N}{N+\beta}|\psi|^\frac{2N}{N+\beta}\,dx\Big)^\frac{N+\beta}{2N}\\
& \leq
C 
\Big( \int_{\operatorname{supp}(\psi)} |g(u_n)-g(u_0)|^{\frac{N(N+2\beta+2)}{(N+\beta)(\beta+2)}}\,dx\Big)^\frac{(N+\beta)(\beta+2)}{N(N+2\beta+2)}
=o_n(1).
\end{align*}
Finally the second integral goes to $0$ due to the weak convergence $I_\beta * G(u_n)\rightharpoonup I_\beta*G(u_0)$ in $L^{\frac{2N}{N-\beta}}(\R^N)+L^{\frac{2N}{(N-2)p-2\beta}}(\R^N)$, since
$g(u_0)\psi \in L^\frac{2N}{N+\beta}(\operatorname{supp}(\psi))
\subset
[L^{\frac{2N}{N-\beta}}(\R^N)+L^{\frac{2N}{(N-2)p-2\beta}}(\R^N)]'$, being
	\begin{align*}
	\int_{\operatorname{supp}(\psi)} |g(u_0)\psi |^\frac{2N}{N+\beta}\, dx
	&\leq
	C  \int_{\operatorname{supp}(\psi)} (|u_0|^{p-1} + |u_0|^\frac{\beta+2}{N-2})^\frac{2N}{N+\beta} |\psi |^\frac{2N}{N+\beta}\, dx\\
	&\leq
	 C \Big(\int_{\operatorname{supp}(\psi)} (|u_0|^{p-1} + |u_0|^\frac{\beta+2}{N-2})^{\frac{N(N+2\beta+2)}{(N+\beta)(\beta+2)}}\, dx\Big)^\frac{2(\beta+2)}{N+2\beta+2}<+\infty.
	\end{align*}
\end{proof}

\section{Proofs of our main results}

Let $X:=\Dr$, or $X:=\Dnr\cap X_\tau$ provided that $N\geq 4$ and $N\neq 5$.
As observed  before, the functional $\cI$ could be also $+\infty$ on $X$. To avoid this problem, for every $n\geq 1$, we introduce the truncated $\cC^1$-functionals
 $\cI_n: X\to \R$ defined by \eqref{defIn}.
 
The functionals $\cI$ and $\cI_n$, $n\geq 1$, satisfy the geometrical assumptions of the  Mountain Pass Theorem. Indeed, we prove the following lemma.
\begin{Lem}\label{lemv0}
	We have:
	\begin{enumerate}[label=(\roman*),ref=\roman*]
		\item \label{MPIi} there exist $\rho, c>0$ such that $\cI(u)\ge c$ and, for every $n\geq 1$, $\cI_n(u)\geq c$ for all $u\in X$ such that $\|\nabla u\|_2 = \rho$;
		\item \label{MPIii} there exists $v_0\in X$ with $\|\nabla v_0\|_2>\rho$ such that $\cI(v_0)<0$ and, for every $n\geq 1$, $\cI_n(v_0)<0$.
	\end{enumerate}
\end{Lem}
\begin{proof}
	We prove this lemma only for $\cI_n$ since similar and easier arguments hold also for $\cI$.
	\\
	The positivity of $G$ and $\vp_n$,  (\ref{h1}) and (\ref{ass0}), the Hardy-Littlewood-Sobolev and Sobolev inequalities imply
	\begin{align*}
	\cI_n(u)
	\geq
	\|\nabla u\|_2^2 - C \int_{\R^N} \left(I_\a * |u|^\frac{N+\a}{N-2}\right)|u|^\frac{N+\a}{N-2}\,dx
	\geq
	\|\nabla u\|_2^2 - C \|u\|_{2^*}^\frac{2(N+\a)}{N-2}
	\geq
	\|\nabla u\|_2^2 - C \|\nabla u\|_{2}^\frac{2(N+\a)}{N-2}.
	\end{align*}
	Since $2<\frac{2(N+\a)}{N-2}$, we get (\ref{MPIi}).\\
	Now let us prove (\ref{MPIii}).\\
	 {\em Case $X=\Dr$.} 
Let $w=s_0 \chi_{B_1}$, where $s_0$ is defined in \eqref{xi1}, then
	\begin{equation*}
	\mathcal{F}(w)
	=
	F^2(s_0) \iint_{B_{1}\times B_{1}}I_\alpha (x-y)\,dxdy>0.
	\end{equation*}
We take now $\psi \in \cC_0^\infty(\RN)$ radial, non-negative, non-increasing with respect to $|x|$, and such that $\psi(x)=s_0$, for $|x|\leq 1$, and $\psi(x)=0$, for $|x|\ge \bar r$, with $\bar r>1$. If $\bar{r}$ is sufficiently close to $1$, using the continuity of $\mathcal{F}$ in $L^{2^*}(\RN)$, we get also
	\begin{equation}\label{posit}
	\mathcal{F}(\psi)
>0.
	\end{equation}	
	We consider first the case	$\alpha>\beta$. 
	\\
	If we set $\psi_\lambda(x):=\psi(x/\lambda)$, $\lambda>0$ and since $0\le \vp_n\le 1$ we have 
	\begin{equation*}
	\int_{\R^N} \vp_n(x)\big(I_\b * G(\psi_\lambda)\big)G(\psi_\lambda)\,dx
	\leq
	\lambda^{N+\b} 	\int_{\R^N} \big(I_\b * G(\psi)\big)G(\psi)\,dx<+\infty.
	\end{equation*}
	So we infer that
	\[
	\cI_n(\psi_\lambda)
	\leq
	\lambda^{N-2}\|\nabla \psi\|_2^2
	-\lambda^{N+\a} \mathcal{F}(\psi)
	+\lambda^{N+\beta} \mathcal{G}(\psi)
	\]
	and	we can conclude considering  $v_0:=\psi_\lambda$ with $\lambda$ large enough, by \eqref{posit}.
	\\
	We now study the case $\alpha=\beta$.
	\\
	If $G(s_0)= 0$, being, by \eqref{h1}, $G$ non-decreasing on $\R_+$, then  $G(\psi(x))=0$ in $\R^N$ and so we can conclude easily as before.
	\\
	If, instead, $G(s_0)\neq 0$, by (\ref{xi1}) we can find $\eps>0$ sufficiently small such that $(1-\eps)F^2(s_0)>G^2(s_0)>0$. Moreover there exists $\bar r>1$  sufficiently close to $1$ such that
	\[
	1<\frac{\dis\iint_{B_{\bar r}\times B_{\bar r}}I_\alpha (x-y)\,dxdy}{\dis\iint_{B_{1}\times B_{1}}I_\alpha (x-y)\,dxdy}<\frac{(1-\eps)F^2(s_0)}{G^2(s_0)}
	\]
and, again by the continuity of $\mathcal{F}$  in $L^{2^*}(\RN)$, 
	\begin{equation*}
		\mathcal{F}(\psi)
	\geq
	(1-\eps)F^2(s_0) \iint_{B_{1}\times B_{1}}I_\alpha (x-y)\,dxdy>0.
	\end{equation*}	
	Therefore,
by the positivity of $G$, we deduce that
\[
	 \mathcal{F}(\psi) - \mathcal{G}(\psi) \geq
	(1-\eps)F^2(s_0) \iint_{B_{1}\times B_{1}}I_\alpha (x-y)\,dxdy
	-
	G^2(s_0) \iint_{B_{\bar r}\times B_{\bar r}}I_\alpha (x-y)\,dxdy
	>0.
\]
	Thus we get
	\[
	\cI_n(\psi_\lambda)
	\leq\lambda^{N-2}\|\nabla \psi\|_2^2
	-\lambda^{N+\a}[ \mathcal{F}(\psi)
	- \mathcal{G}(\psi) ],
	\]
	we can conclude again considering  $v_0:=\psi_\lambda$ with $\lambda$ large enough.\\	
{\em Case $X=\Dnr\cap X_\tau$.}\\
We take $\eps\in (0,1/2)$ and any odd and smooth function $\eta:\R\to [-1,1]$ such that $\eta(s)=1$ for $s\geq 1/2$ and $\eta(s)=0$ for $s\leq 1/2-\eps$. Then we define $\widetilde{\psi}(x)=\eta(|x_1|-|x_2|)\psi(x)$ for $x=(x_1,x_2,x_3)\in\R^M\times\R^M\times\R^{N-2M}$, with the same $\psi$ as before. Observe that $\widetilde{\psi}\in X$. 
Moreover, arguing as in the previous case, we can find $\bar{r}>1$, sufficiently close to $1$,  and $\eps>0$, sufficiently close to $0$, such that, using the continuity of $\mathcal{F}$  in $L^{2^*}(\RN)$, 
$$\mathcal{F}(\widetilde\psi)
\geq
\frac 12 F^2(s_0) \iint_{B_{1}\times B_{1}\cap \{|x_1|\geq |x_2|+1/2,|y_1|\geq |y_2|+1/2\}}I_\alpha (x-y)\,dxdy>0.$$
Then we argue similarly as in case $X=\Dr$.
\end{proof}

Let
\begin{eqnarray*}
\Gamma:=\left\{ \gamma\in\cC\left([0,1],X\right): \gamma(0)=0\hbox{ and }\gamma(1)=v_0\right\}
\end{eqnarray*}
and
\begin{eqnarray*}
c_{\cI_n}&:=&\inf_{\gamma\in\Gamma}\sup_{t\in [0,1]}\cI_n(\gamma(t)),
\quad c_{\cI}:=\inf_{\gamma\in\Gamma}\sup_{t\in [0,1]}\cI(\gamma(t)).
\end{eqnarray*}

Our aim is to find a sequence $\{u_n\}\subset X$ such that $\cI_n(u_n)=c_{\cI_n}$ and $\cI_n'(u_n)\to 0$, as $n \to +\infty$. However, due to the general assumptions on $F$ and $G$, it is not easy to prove the boundedness of such sequence. Therefore, inspired by \cite{HIT,jj}, we introduce  the functional $\cJ:\R\times X\to \R\cup \{+\infty\}$
\[
\cJ(\sigma, u)
:=\cI(u(e^{-\s}\cdot))=
e^{(N-2)\s}\|\nabla u\|_2^2
-e^{(N+\a)\s} \mathcal{F}(u)
+e^{(N+\b)\s} \mathcal{G}(u),
\]
and, for every $n\geq 1$, the $\cC^1$-functionals $\cJ_n:\R\times X\to \R$
\[
\cJ_n(\sigma, u)
:=\cI_n(u(e^{-\s}\cdot))=
e^{(N-2)\s}\|\nabla u\|_2^2
-e^{(N+\a)\s}\mathcal{F}(u)
+e^{(N+\b)\s}\int_{\R^N}\vp_n(e^\s x)\big(I_\beta\ast G(u)\big)G(u)\,dx.
\]
Since the functional  $\cJ_n$ is non-autonomous, contrary to \cite{HIT,jj}, we have to deal with an additional term which appears in the calculus of the gradient of $\cJ_n$, see Proposition \ref{PropPS} below.

Let
\begin{eqnarray*}
\Sigma:=\left\{ (\sigma,\gamma)\in\cC\left([0,1],\R\times X\right): \big(\sigma(0),\gamma(0)\big)=(0,0)\hbox{ and }\big(\sigma(1),\gamma(1)\big)=(0,v_0)\right\}
\end{eqnarray*}
and
\begin{eqnarray*}
c_{\cJ_n}:=\inf_{(\sigma,\gamma)\in\Sigma}\sup_{t\in [0,1]}\cJ_n\big(\sigma(t),\gamma(t)\big),\quad c_{\cJ}:=\inf_{(\sigma,\gamma)\in\Sigma}\sup_{t\in [0,1]}\cJ\big(\sigma(t),\gamma(t)\big).
\end{eqnarray*}

As observed in \cite[Lemma 4.1]{HIT}, using the relation, respectively, between $\cI$ and $\cJ$ and $\cI_n$ and $\cJ_n$, we have that 
\begin{equation}\label{cicij}
c_{\cI}=c_{\cJ}, \qquad c_{\cI_n}=c_{\cJ_n}.
\end{equation}

Since, for any $n\in \N$, $\cJ_{n}\le \cJ_{n+1}\le \cJ$, we have that the sequence $\{c_{\cJ_{n}}\}$ is increasing and bounded from above by $c_\cJ$, and so there exists $\bar c>0$ such that $c_{\cJ_{n}}\to \bar c$, as $n \to +\infty$.


\begin{Prop}\label{PropPS}
There is a sequence $\{(\sigma_n,u_n)\}$ in $\R\times X$ such that
\begin{enumerate}[label=(\roman*),ref=\roman*]
	\item \label{3.2i} $|\cJ_n(\sigma_n,u_n)-\bar c|=o_n(1)$;
	\item \label{3.2ii} $|\sigma_n|=o_n(1)$;
	\item \label{3.2iii} $\|\nabla \cJ_n(\sigma_n,u_n)\|=o_n(1)$;
	\item \label{3.2iv} $\{u_n\}$ is bounded in $X$.
\end{enumerate}
\end{Prop}
\begin{proof}
In view of \eqref{cicij}, for any $n\geq 1$ we find $\gamma_{k,n}\in\Gamma$ such that
$$\sup_{t\in [0,1]}\cI_{k}(\gamma_{k,n}(t))\leq c_{\cJ_{k}}+\frac{1}{n}$$
and, for sufficiently large $k$,  
$$|c_{\cJ_k}-\bar c|\leq \frac{1}{n}$$
also holds. 
Therefore, passing to a subsequence with a diagonalization argument, we may assume that there exists $\gamma_{n}\in\Gamma$ (hence $(0,\gamma_{n})\in\Sigma$) such that
$$\sup_{t\in [0,1]}\cJ_{n}(0,\gamma_{n}(t))\leq c_{\cJ_{n}}+o_n(1) \quad \hbox{and}\quad  |c_{\cJ_n}-\bar c|\leq o_n(1).$$
Thus, by \cite[Theorem 2.8]{Willem}, for any $n\ge 1$
there is $(\sigma_{n},u_{n})\in \R\times X$ such that (\ref{3.2i})--(\ref{3.2iii}) hold.\\
Since $\cJ_{n}(\sigma_{n},u_{n})=\bar c+o_n(1)$ and $\de_\s \cJ_n(\sigma_{n},u_{n})=o_n(1)$, we have
\begin{align*}
&\left(1-\frac{N-2}{N+\a}\right)e^{(N-2)\s_{n}}\int_{\R^N}|\nabla u_{n}|^2\, dx
+\left(1-\frac{N+\b}{N+\a}\right)e^{(N+\b)\s_{n}}\int_{\R^N}\vp_{n}(e^{\sigma_{n}}x) \big(I_\beta\ast G(u_{n})\big)G(u_{n})\,dx\\
&\qquad-\frac{1}{N+\a}e^{(N+\b)\s_{n}}\int_{\R^N}\big(\nabla\vp_{n}(e^{\sigma_{n}}x)\cdot e^{\sigma_{n}} x \big)\big(I_\beta\ast G(u_{n})\big)G(u_{n})\,dx=\bar c+o_n(1).
\end{align*}
Since the cut-off functions $\vp_{n}$ are decreasing with respect to the radius, we have that $\n \vp_{n}(x)\cdot x\le 0$, for any $x\in \RN$ and so, 
being $\a\ge \b$, we infer that $\{u_{n}\}$ is a bounded sequence in $X$. 
\end{proof}

We can now conclude the proof of our main theorems.

\begin{proof}[Proof of Theorems \ref{ThMain} and \ref{ThMain2}]
Let $\{(\sigma_n,u_n)\}$ in $\R\times X$ be the sequence found in Proposition \ref{PropPS}. Then
there exists $u_0\in X$ such that $u_n \rightharpoonup  u_0$ weakly in $X$ and a.e. on $\R^N$. By Lemma \ref{le:BL1} and Lemma \ref{pippa},  for any $\psi\in\cC_0^{\infty}(\R^N)$, we have that
\begin{equation*}
\int_{\R^N}\nabla u_0 \cdot \n \psi\, dx
=\int_{\R^N} \big(I_\a\ast F(u_0)\big)f(u_0)\psi\,dx
-\int_{\R^N} \big(I_\beta\ast G(u_0)\big)g(u_0)\psi\,dx.
\end{equation*}
So we have that $u_0$ is a weak solution of \eqref{eq}. We will prove that $u_0\neq 0$.
\\
Observe that, by Proposition \ref{PropPS}, since $\{u_n\}$ is bounded in $X$ and $\de_u \cJ_n(\sigma_n,u_n)[u_n]=o_n(1), $ we deduce that there exists $C>0$ such that, for any $n\ge 1$, 
\[
\int_{\R^N} \vp_n(x) \big(I_\b\ast G(u_n)\big)g(u_n)u_n\,dx\le C.
\]
Therefore, by Fatou's Lemma
\begin{equation}\label{u0l1}
\int_{\R^N}\big(I_\b\ast G(u_0)\big)g(u_0)u_0\,dx\le \liminf_{n} \int_{\R^N} \vp_n(x) \big(I_\b\ast G(u_n)\big)g(u_n)u_n\,dx\le C.
\end{equation}
For any  $m\ge 1$, let 
\[
\psi_m(x)=
\begin{cases}
1 & \hbox{if }|x|\le m,
\\[3mm]
\dis \frac{2m -|x|}{m} & \hbox{if  }m\le |x|\le 2m,
\\[3mm]
0 & \hbox{if }|x|\ge 2m.
\end{cases}
\] 
Observe that, for any $m\ge 1$, we have that $\psi_m u_0$ belongs to $X$.
Note that $\psi_mu_0$ has a compact support and $\de_u \cJ_n(\sigma_n,u_n)[\psi_m u_0]=o_n(1)$. Therefore,  arguing as in Lemma \ref{le:BL1} and in Lemma \ref{pippa}, passing to the limit as $n \to +\infty$, we have that for any $m\ge 1$
\begin{equation}\label{almsol}
\int_{\R^N}\nabla u_0 \cdot \n (\psi_m u_0)\, dx
=\int_{\R^N} \big(I_\a\ast F(u_0)\big)f(u_0)\psi_m u_0\,dx
-\int_{\R^N} \big(I_\beta\ast G(u_0)\big)g(u_0)\psi_m u_0\,dx.
\end{equation}
Being $u_0\in X$, we have
\begin{equation}\label{u0psi}
\begin{split}
&\left|\int_{\R^N}\nabla u_0 \cdot \n (\psi_m u_0)\, dx-
\int_{\R^N}|\nabla u_0|^2\, dx\right|\\
&\qquad\le
\int_{\R^N}|\nabla u_0 |^2 |\psi_m -1|\, dx
+\int_{\R^N}|\nabla u_0||u_0|  |\n \psi_m|\, dx \\
&\qquad\le
\int_{B_m^c}|\nabla u_0 |^2 \, dx
+\Big(\int_{A_m}|\nabla u_0|^2\, dx\Big)^\frac 12
\Big(\int_{A_m}|u_0|^{2^*}\, dx\Big)^\frac 1{2^*}
\Big(\int_{A_m}|\nabla \psi_m|^N\, dx\Big)^\frac 1N
\\
&\qquad\le
\int_{B_m^c}|\nabla u_0 |^2 \, dx
+ C \Big(\int_{B_{m}^c}|\nabla u_0|^2\, dx\Big)^\frac 12
\Big(\int_{B_{m}^c}|u_0|^{2^*}\, dx\Big)^\frac 1{2^*}\\
&\qquad=o_m(1),
\end{split}
\end{equation}
where $A_m:= B_{2m}\setminus B_m$.\\
Moreover, observe that 
\[
\big(I_\a\ast F(u_0)\big)f(u_0)\psi_m u_0 \to \big(I_\a\ast F(u_0)\big)f(u_0) u_0,\quad
\text{ a.e. in } \RN, \text{ as } m \to +\infty,
\]
and
\[
\left|\big(I_\a\ast F(u_0)\big)f(u_0)\psi_m u_0\right|\le \left|\big(I_\a\ast F(u_0)\big)f(u_0) u_0 \right|\in L^1(\RN).
\]
Thus, by the Dominated  Convergence Theorem, we have that
\begin{equation}\label{Fpsi}
\lim_m \int_{\R^N} \big(I_\a\ast F(u_0)\big)f(u_0)\psi_m u_0\,dx=
\int_{\R^N} \big(I_\a\ast F(u_0)\big)f(u_0)u_0\,dx.
\end{equation}
Analogously, we have also that
\[
\big(I_\b\ast G(u_0)\big)g(u_0)\psi_m u_0 \to \big(I_\b\ast G(u_0)\big)g(u_0) u_0, \quad
\text{ a.e. in } \RN, \text{ as } m \to +\infty,
\]
and, using \eqref{u0l1},
\[
0\le \big(I_\b\ast G(u_0)\big)g(u_0)\psi_m u_0\le \big(I_\b\ast G(u_0)\big)g(u_0) u_0 \in L^1(\RN).
\]
Again the Dominated  Convergence Theorem implies
\begin{equation}\label{Gpsi}
\lim_m \int_{\R^N} \big(I_\b\ast G(u_0)\big)g(u_0)\psi_m u_0\,dx=
\int_{\R^N} \big(I_\b\ast G(u_0)\big)g(u_0)u_0\,dx.
\end{equation}
Therefore, by \eqref{almsol}, \eqref{u0psi}, \eqref{Fpsi} and \eqref{Gpsi}, we have
\begin{equation*}
\int_{\R^N}|\nabla u_0|^2\, dx
=\int_{\R^N} \big(I_\a\ast F(u_0)\big)f(u_0)u_0\,dx
-\int_{\R^N} \big(I_\beta\ast G(u_0)\big)g(u_0)u_0\,dx.
\end{equation*}
By Lemma \ref{le:BLnew} and \eqref{u0l1},
since  $\de_u \cJ_n(\sigma_n,u_n)[u_n]=o_n(1)$, we infer that
\begin{align*}
\limsup_{n} \int_{\R^N}|\nabla u_n|^2\, dx
&=\limsup_{n} \left[\int_{\R^N} \big(I_\a\ast F(u_n)\big)f(u_n)u_n\,dx
-\int_{\R^N}\vp_n(x)  \big(I_\beta\ast G(u_n)\big)g(u_n)u_n\,dx\right]
\\
&\le \int_{\R^N} \big(I_\a\ast F(u_0)\big)f(u_0)u_0\,dx
-\int_{\R^N} \big(I_\beta\ast G(u_0)\big)g(u_0)u_0\,dx
=\int_{\R^N}|\nabla u_0|^2\, dx.
\end{align*}
This implies that $u_n\to u_0$ strongly in $X$.
Thus, since $\cJ_n(\sigma_n,u_n)\to \cI(u_0)$, we have that $\cI(u_0)=\bar c>0$ and so $u_0$ is a nontrivial weak solution of \eqref{eq}.
\end{proof}


\begin{Rem}
By the inspection of the proof of Theorems \ref{ThMain} and \ref{ThMain2}, we deduce that $c_{\cI_{n}}=c_{\cJ_{n}}$ are attained, for any $n\geq 1$.
\end{Rem}

\begin{altproof}{Theorem \ref{ThMain3}}
The proof is a slight modification of our previous arguments. Here we just want to comment \eqref{xi1'}. The change of assumption in the different cases is due to the scaling properties of the functional $\K_\omega$. 
	Indeed, setting $u_\lambda(x):=u(x/\lambda)$, for $\lambda>0$,  when $\alpha=\beta$ we have
	\[
	\K_\omega(u_\lambda)=	\lambda^{N-2}\|\nabla u\|_2^2
	+\omega\lambda^{N}\|u\|_2^2
	-\lambda^{N+\a} \big(\mathcal{F}(u)
	- \mathcal{G}(u)\big).
	\]
	Thus, to show the Mountain Pass geometry, if $\alpha=\beta>0$, we can proceed as in Lemma \ref{lemv0}, but if $\alpha=\beta=0$ (the local case), we need a stronger condition, namely we need to take into account the term $\omega s_0^2$ in order to show that $\K_\omega(u_\lambda)<0$ for large $\lambda$ (see also \cite{BL1}).
\end{altproof}

\begin{altproof}{Corollary \ref{th:cubicquintic}}
Item (\ref{a14}) follows from Theorem \ref{ThMain3}.\\
To prove (\ref{b14}), observe that, only in the local case $\alpha=\beta=0$, $\omega_0$ is finite. Thus, in such a case, if $\omega\geq \omega_0$, then $F^2(s)-G^2(s)-\omega_0s^2\leq 0$ for $s\in\R$ and  there are no nontrivial solutions (see e.g. \cite{BL1}).\\
If, instead, $\omega\leq 0$,
similarly as in \cite[Theorem 3]{MVTrans}, if $u\in H^1(\R^N)$ solves \eqref{eq2} with \eqref{eq:CubicQuintic}, then we obtain the following Pohozaev identity
\[
\|\nabla u\|_2^2
=-\omega\frac{N}{N-2}\|u\|_2^2
+\frac{N+\alpha}{q(N-2)}\int_{\R^N}\big(I_\alpha\ast |u|^{q}\big)|u|^{q}\,dx
-\int_{\R^N}\big(I_\beta\ast |u|^{\frac{N+\beta}{N-2}}\big) |u|^{\frac{N+\beta}{N-2}}\,dx
\] 
and, taking into account  $\K_\omega'(u)[u]=0$, i.e.
\[
\|\nabla u\|_2^2
=-\omega\|u\|_2^2
+\int_{\R^N}\big(I_\alpha\ast |u|^{q}\big)|u|^{q}\,dx
-\int_{\R^N}\big(I_\beta\ast |u|^{\frac{N+\beta}{N-2}}\big) |u|^{\frac{N+\beta}{N-2}}\,dx,
\]
we infer that $u=0$.
\end{altproof}


\subsection*{Acknowledgments}
The authors wish to thank the anonymous referees for their valuable suggestions and remarks.

\end{document}